\documentclass[letterpaper,10pt]{article}

\usepackage[margin=0.5in]{geometry}

\usepackage{fullpage}
\usepackage{enumerate}
\usepackage{authblk}
\usepackage[colorinlistoftodos]{todonotes}
\usepackage{verbatim}%comments out anything between \begin{comment} and \end{comment}
\usepackage{section, amsthm, textcase, setspace, amssymb, lineno, amsmath, amssymb, amsfonts, latexsym, fancyhdr, longtable, ulem, mathtools}
\usepackage{epsfig, graphicx, pstricks,pst-grad,pst-text,tikz, tkz-berge,colortbl}
\usepackage{epsf}
\usepackage{graphicx, color}
\usepackage{float}
\usepackage[rflt]{floatflt}
\usepackage{amsfonts}
\usepackage{latexsym,enumitem}
\usetikzlibrary{fit,matrix,positioning}
\usepackage{pdflscape}
\usetikzlibrary{decorations.pathreplacing}
\usepackage{mathrsfs}
\usepackage{makecell}

\definecolor{darkgreen}{rgb}{0, 0.5, 0}

\newtheorem{theorem}{Theorem}

\newtheorem{Ex}{Example}

\newtheorem*{theorem*}{Theorem}
\newtheorem{remark}{Remark}

\newcommand{\eps}{\varepsilon}
\newcommand{\F}{\mathbb{F}}

\begin{document}

\title{The evolution of the spectrum of a Frobenius Lie algebra under deformation}

\author[*]{Vincent E. Coll, Jr.}
\author[**]{Nicholas Mayers}
\author[*]{Nicholas Russoniello}

\affil[*]{Department of Mathematics, Lehigh University, Bethlehem, PA, 18015}
\affil[**]{Department of Mathematics, Milwaukee School of Engineering, Milwaukee, WI, 53202}

\maketitle

%\linenumbers

\begin{abstract}
\noindent

\noindent
The category of Frobenius Lie algebras is stable under deformation, and here we examine explicit infinitesimal deformations of four and six dimensional Frobenius Lie algebras with 
the goal of understanding if 
the spectrum of a Frobenius Lie algebra can evolve under deformation. It can.

%While it is known that category of Frobenius Lie algebras is stable under deformation, there are no explicit examples of infinitesimal deformations of a  Frobenius Lie algebras -- or even an existential proof that such exist. Here we present a four-dimensional example of such an algebra.  We examine the effect of deformation on the spectrum of this algebra as well as make use of the spectrum of the deformed algebras to show that for all values of the deformation parameter these algebras are non-isomorphic.
\end{abstract}



%\tableofcontents
\section{Introduction}

A Lie algebra ($\mathfrak{g}, [-,-]$) is \textit{Frobenius} if there exists a linear functional $F\in \mathfrak{g}^*$ such that the natural map $\eta :\mathfrak{g} \rightarrow \mathfrak{g}^*$ defined by $x \mapsto  F[x,-]$ is an isomorphism. Such an $F$ is called a \textit{Frobenius functional}.  The set of Frobenius functionals of a Frobenius Lie algebra $\mathfrak{g}$ is, in general, quite large; forming an open subset of $\mathfrak{g}^{*}$ in the Zariski and Euclidean topologies (see \textbf{\cite{Ooms2}} and \textbf{\cite{DK}}).

Frobenius Lie algebras were introduced in the 1970's by Ooms who showed, in particular, that the universal enveloping algebra $U(\mathfrak{g})$ admits a faithful simple representation when $\mathfrak{g}$ is Frobenius (see \textbf{\cite{Ooms2}}). Such algebras also have applications in invariant theory and the geometry of coadjoint orbits in $\mathfrak{g}^*$ (see \textbf{\cite{orbit}}). Deformation theorists are interested in Frobenius Lie algebras because each provides a solution to the classical Yang-Baxter equation, which in turn quantizes to a universal deformation formula, i.e., a Drinfel'd twist which deforms any algebra which admits an action of $\mathfrak{g}$ by derivations (see \textbf{\cite{twist}}).

If $\mathfrak{g}$ is Frobenius and $F\in \mathfrak{g}^*$ is a Frobenius functional, then the inverse image of $F$ under the mapping
$\eta$ is called a \textit{principal element} of $\mathfrak{g}$ and will be denoted $\widehat{F}$ (see \textbf{\cite{Prin}}).  It is the unique element of $\mathfrak{g}$ such that 
$$
F\circ \mathrm{ad}~\widehat{F}= F([\widehat{F},-]) = F.  
$$

%if there exists a linear functional $F\in \mathfrak{g}^*$ such that $(x,y)\rightarrow F([x,y])$ is non-degenerate.  

%Frobenius algebras are of  special interest in deformation and quantum group theory stemming from their connection with the classical Yang-Baxter equation (see \textbf{\cite{G1}} and \textbf{\cite{G2}}).  More specifically, an index-realizing functional is called \textit{regular}, and a regular functional $F$ on a Frobenius Lie algebra $\mathfrak{g}$ is called a \textit{Frobenius functional}; equivalently, $B_F(-,-)$ is non-degenerate.  Suppose $B_F(-,-)$ is non-degenerate and let $[F]$ be the matrix of $B_F(-,-)$ relative to some basis 
%$\{x_1,\dots,x_n  \}$ of $\mathfrak{g}$.  In \textbf{\cite{BelDrin}}, Belavin and Drinfeld showed that   
%\[
%\sum_{i,j}[F]^{-1}_{ij}x_i\wedge x_j
%\]
%\noindent
%is the infinitesimal of a \textit{Universal Deformation Formula} (UDF) based on $\mathfrak{g}$.  A UDF based on $\mathfrak{g}$ can be used to deform the universal enveloping algebra of $\mathfrak{g}$ and also the function space on any Lie group which contains $\mathfrak{g}$ in its Lie algebra of derivations.

%We call such an $F$ a \textit{Frobenius linear form} and the natural map $\mathfrak{g} \rightarrow \mathfrak{g}^*$ defined by $x \mapsto  F[x,-]$ is an isomorphism.  The image of $F$ under the inverse of this map is called a \emph{principal element} of $\mathfrak{g}$ and will be denoted $\hat{F}$ \textbf{\cite{Prin}}.  It is the unique element of $\mathfrak{g}$ such that 
%$$
%F\circ \mathrm{ad}~\hat{F}= %F([\hat{F},-]) = F.  
%$$
\noindent In \textbf{\cite{Ooms2}}, Ooms established that the spectrum of the adjoint of a principal element of a Frobenius Lie algebra is independent of the principal element chosen to compute it (see also \textbf{\cite{G2}}).  Consequently, we can unambiguously refer to the $spectrum$ of $\mathfrak{g}$ as the spectrum of the adjoint representation of any Frobenius functional $F\in \mathfrak{g}^*$.

Since the category of Frobenius Lie algebras is stable under deformation -- the index can only decrease under deformation (see \textbf{\cite{AM}}) -- we come to the motivating question of this article.

\bigskip
\noindent
 \textit{Q:  Does the spectrum of a Frobenius Lie algebra evolve under the deformation of the underlying algebra?}

\bigskip
\noindent

To address this answer one requires easy examples of infinitesimal deformations of Frobenius Lie algebras.  Such examples are sparse (see \textbf{\cite{contact}}).

Here, following  Czikos’s and Verhoczkis' classification of four and six-dimensional Frobenius Lie algebras (\textbf{\cite{CV}}, cf. Tables 1 and 2), we investigate the infinitesimal deformation theory of these algebras to find that many of these Frobenius Lie algebras can be deformed.  A particularly rich example is provided by a certain one of these four-dimensional algebras, $\Phi^\prime$, for which we provide detailed cohomological and spectral calculations. The calculations are routine but potentially instructive. 

%This algebra three inequivalent infinitesimal deformations; two of which are jump deformations.  Of more primary interest, however, is that we see that generally, the spectrum does evolve under the deformation of the underlying algebra.  For $\Phi^\prime$, we can use the spectrum to show that the third deformation defines a one-parameter family of non-isomorphic Lie algebras.

% Predicated on the 2007 paper of Csikos and Verhoczki \textbf{\cite{CV}} where, Frobenius Lie algebras of dimension four and six are classified (see Table~\ref{4D}).  Here we examine the deformations of these algebras . 
 %A particularly rich example is provided by  a certain one of these four-dimensional algebras, which, following  \textbf{\cite{CV}}

%n \textbf{\cite{Prin}}, Gerstenhaber and Giaquinto introduce the notion of the ``spectrum" of a principal element of a Frobenius Lie algebra and establish that the category of Frobenius Lie algebras is stable under deformation (cf.  \textbf{\cite{AM}}).  This spectrum is independent of the principal element used to compute it. It follows that the spectrum is an algebraic invariant of the algebra, which we simply call the \textit{spectrum} of the algebra.

%The original motivation for this paper was to try to understand how -- or if -- the spectrum of a Frobenius Lie algebra evolved under the deformation of the underlying algebra. 

 The structure of the paper is as follows. In Section \ref{defcoh}, we recall the well-known infinitesimal deformation theory of Nijenhuis and Richardson (see \textbf{\cite{F,Fox}}), wherein the deformation of a Lie algebra is controlled by the graded Chevalley-Eilenberg complex (see \textbf{\cite{Sasha}}). In Section \ref{class}, we present the classification of four-dimensional Frobenius Lie algebras by detailing in Table \ref{4D} the commutator relations, the dimension of the 
 second and third cohomology groups of the Lie algebra with coefficients in the Lie algebra (the case of interest for deformation theory), the spectrum, and whether or not deformations of the algebra exist.
  In Section \ref{sec:deformation}, we provide detailed deformation theory calculations associated with $\Phi^\prime$ yielding Table \ref{table:def} where we provide the full deformation and spectral values at deformation parameter instance $t$, for all deformations of the four-dimensional Lie algebras here considered.
  In Section 5, a short Epilogue provides connections between this paper's examples and topical ``spectral'' research, along with announcements of new results that will be appearing and will be of interest to researchers in this area.  In Appendix A, Table \ref{table:six} is the analogue of Table \ref{table:def} for the six-dimensional Frobenius Lie algebras in Czikos's and Verhoczki's classification.  The calculations are similar to those in the four-dimensional case but more tedious for the parametrized families.  We provide the results of our cohomological calculations for one such parametrized family $\Phi_{6,12}(\xi)$  
  (see Example \ref{ex:6dim}).  We find this example interesting because 
  for all values of 
  $\xi$, $\dim H^2(\Phi_{6,12}(\xi)\neq 0$ and only eight values of $\xi$ where 
    $\dim H^3(\Phi_{6,12}(\xi),\Phi_{6,12}(\xi))\neq 0$.  Even so, all infinitesimals are unobstructed.
    
    For the non-parametrized six-dimensional families, we provide the number of inequivalent deformations. As with the four-dimensional Lie algebras, all deformations are linear.

  %for almost all values of $\xi$ both $H^2(\Phi_{6,12}(\xi),\Phi_{6,12}(\xi))$ and $H^3(\Phi_{6,12}(\xi),\Phi_{6,12}(\xi))$ are non-zer0, 

%%%%%%%%%%%%%%%%%%%%%%%%%%%%%%%%%%%%%%%%%
\section{Deformation theory}\label{defcoh}
%%%%%%%%%%%%%%%%%%%%%%%%%%%%%%%%%%%%%%%%%

Let $(\mathfrak{g}, [~,~])$ be a Lie algebra over an algebraically closed field $\F$, where char $\F \neq 0$. A formal one-parameter deformation of $\mathfrak{g}$ is a power series
\begin{equation}
[g,h]_t=[g,h]+\sum_{k\geq 1}\alpha_k(g,h)t^k,
\end{equation}
where $\alpha_k\in$ $\rm{HOM}_\F(\Lambda^2\mathfrak{g},\mathfrak{g})=C^2(\mathfrak{g},\mathfrak{g})$.  The latter refers to the standard Chevalley-Eilenberg cochain complex $(C^{\bullet}(\mathfrak{g},\mathfrak{g}),\delta)$ of $\mathfrak{g}$ with coefficients in the adjoint representation of $\mathfrak{g}$. Here, $C^n(\mathfrak{g},\mathfrak{g})$ consists of forms $F^n:\bigwedge_{i=1}^n\mathfrak{g}\to \mathfrak{g}$ satisfying $\delta^2 F^n=0$, where the coboundary operator $\delta$ is defined by 
\begin{align}
\delta F^n(g_1,\hdots,g_{n+1})=\sum_{i=1}^{n+1}(-1)^{i+1}[&g_i, F^n(g_1,\hdots,\hat{g}_i,\hdots,g_{n+1})]\nonumber \\ 
&+\sum_{1\le i<j\le n+1}(-1)^{i+j}F^n([x_i,x_j],x_1,\hdots,\hat{x}_i,\hdots,\hat{x}_j,\hdots,x_{n+1}).\nonumber
\end{align}

In this setting, $Z^n(\mathfrak{g},\mathfrak{g})=\ker(\delta)\cap C^n(\mathfrak{g},\mathfrak{g})$, $B^n(\mathfrak{g},\mathfrak{g})=Im(\delta)\cap C^n(\mathfrak{g},\mathfrak{g})$, and $H^n(\mathfrak{g},\mathfrak{g})=Z^n(\mathfrak{g},\mathfrak{g})/ B^n(\mathfrak{g},\mathfrak{g})$.  These comprise, respectively, the \textit{n-cocycles}, \textit{n-coboundaries}, and $n^{th}$ \textit{cohomology group} of $\mathfrak{g}$ with coefficients in $\mathfrak{g}$.  Of course, one requires that the deformation remains in the category of Lie algebras so that the Jacobi identity for $[~,~]_t$ is satisfied for all values of $t$. This is equivalent 
to the sequence of relations

\begin{equation}\label{eqn:massey}
\delta \alpha_k = -\frac{1}{2}\sum_{i=1}^{k-1} [\alpha_i,\alpha_{k-i}],
\end{equation}

\noindent
where 
\begin{equation}\label{eqn:mprod}
\begin{split}
& [\beta,\gamma](g_1,\hdots,g_{p+q-1})=\sum_{1\le i_1<\hdots<i_q\le p+q-1}(-1)^{i_1+\hdots+i_q-\frac{q(q+1)}{2}}\beta(\gamma(g_{i_1},\hdots,g_{i_q}),g_1,\hdots\hat{g}_{i_1}\hdots\hat{g}_{i_q}\hdots,g_{p+q-1}) \\ & +(-1)^{pq+p+q}\sum_{1\le j_1<\hdots<j_p\le p+q-1}(-1)^{j_1+\hdots+j_p-\frac{p(p+1)}{2}}\gamma(\beta(g_{j_1},\hdots,g_{j_p}),g_1,\hdots\hat{g}_{j_1}\hdots\hat{g}_{j_p}\hdots,g_{p+q-1}).
\end{split}
\end{equation}

\noindent In particular, when $p,q=2$ equation (\ref{eqn:mprod}) becomes $$[\beta,\gamma](g_1,g_2,g_3)=\sum_{1\le i_1<i_2\le 3}(-1)^{i_1+i_2-3}\beta(\gamma(g_{i_1},g_{i_2}),g_1,\hdots\hat{g}_{i_1}\hdots\hat{g}_{i_3}\hdots,g_3)$$ $$+\sum_{1\le j_1<j_2\le 3}(-1)^{j_1+j_2-3}\gamma(\beta(g_{j_1},g_{j_2}),g_1,\hdots\hat{g}_{j_1}\hdots\hat{g}_{j_2}\hdots,g_3).$$

Two deformations $[g,h]_t,[g,h]_t^\prime$ are called \textit{equivalent} if there exists a formal one-parameter family $\{\phi_t\}$ of linear transformations of $\mathfrak{g},$ $$\phi_t(g)=g+\sum_{k\geq 1}\beta_k(g)t^k,$$ such that \begin{equation}\label{eqn:equiv}
[g,h]_t^\prime=\phi_t^{-1}[\phi_t(g),\phi_t(h)]_t.
\end{equation}

It is easy to see that equivalent deformations have cohomologous infinitesimals:  that is, $\alpha_1 - \alpha_1^{\prime} = \delta \beta$ in (\ref{eqn:equiv}).  Therefore, $\alpha_1$ (more precisely, its equivalence class $[\alpha_1]\in H^2(\mathfrak{g},\mathfrak{g})$) is called the \textit{infinitesimal} of the deformation.  So, up to equivalence,  the infinitesimal deformations of $\mathfrak{g}$ may be regarded as elements of 
$H^2 ( \mathfrak{g}, \mathfrak{g})$ with the obstructions to their propagation to higher-order deformations lying in $H^3 ( \mathfrak{g}, \mathfrak{g})$.  If each element of $H^2 ( \mathfrak{g}, \mathfrak{g})$ is obstructed, then $\mathfrak{g}$
is called \textit{rigid}, and if $H^2 ( \mathfrak{g}, \mathfrak{g})=0$ then 
$\mathfrak{g}$ is said to be  \textit{absolutely rigid}.

As it happens, all the deformations here considered are linear, so we require only the terms corresponding to $k=1$ and 2 in (\ref{eqn:massey}):
$$
\delta\alpha_1 =0\text{ and} ~~\delta \alpha_2=- \frac{1}{2}[\alpha_1,\alpha_1].$$
A \textit{jump deformation} of $\mathfrak{g}$ is one such that all specializations, as $t$ varies over $\F,$ are all isomorphic except perhaps, the specialization to t = 0, which must be isomorphic to $\mathfrak{g}$ itself.

%%%%%%%%%%%%%%%%%%%%%%%%%%%%%%%%%%%%%%%%%%%%%%%
\section{Classification}\label{class}
%%%%%%%%%%%%%%%%%%%%%%%%%%%%%%%%%%%%%%%%%%%%%%%

The following table contains information about the isomorphism classes of four-dimensional Frobenius Lie algebras over a field  $\mathbb{F}$, where char $\F \neq 2.$  Following the notation of \textbf{\cite{CV}}, there is a single four-dimensional Lie algebra $\Phi^\prime$, and two families of four-dimensional Lie algebras $\Phi^{\prime \prime}$ and $\Phi^{\prime \prime \prime}$, parametrized by $\Delta \in \F$ and $0\neq \eps \in \F$, respectively. Note the distinguished $\Delta$ value of 0, which affects the dimension of the second cohomology group.

\begin{table}[H]
\begin{center}
\renewcommand\arraystretch{3}
\setlength\doublerulesep{0pt}
\begin{tabular}{|p{1.5cm}|||p{4.2cm}|p{1.3cm}|p{1.3cm}|p{2.5cm}|p{2.3cm}|}
\hline
\centering $\mathfrak{g}$ & \centering Commutator Relations & \centering dim $H^2$ & \centering dim $H^3$ & \centering Spectrum & \makecell{Deformation} \\
 \hline\hline\hline
 \centering $\Phi'$ & \makecell[l]{$[e_1,e_4]=[e_2,e_3]=-e_1,$ \\  $[e_2,e_4]=-\frac{1}{2}e_2,$ \\   $[e_3,e_4]=-\frac{1}{2}e_3$} &  \centering 3  & \centering 0 & \centering $\left\{0,\frac{1}{2},\frac{1}{2},1\right\}$ & \makecell{Yes}\\
\hline
\centering $\Phi''(0)$ & \makecell[l]{$[e_1,e_4]=[e_2,e_3]=-e_1,$ \\  $[e_2,e_4]=-e_3,$ \\ $[e_3,e_4]=-e_3$} & \centering 2 & \centering 1 & \centering $\left\{0,0,1,1\right\}$ & \makecell{Yes} \\
\hline
 \centering $\Phi''(\Delta)$ & \makecell[l]{$[e_1,e_4]=[e_2,e_3]=-e_1,$ \\  $[e_2,e_4]=-e_3,$ \\ $[e_3,e_4]=-e_3+\Delta e_2$} & \centering 1 & \centering 0 & \centering $\left\{0,1,\frac{1\pm\sqrt{1-4\Delta}}{2}\right\}$ & \makecell{Yes} \\
\hline
\centering $\Phi'''(\varepsilon)$ & \makecell[l]{$[e_1,e_3]=[e_2,e_4]=-e_1,$ \\  $[e_1,e_4]=\varepsilon e_2,$ \\ $[e_2,e_3]=-e_2$} & \centering 0 & \centering 0 & \centering $\{0,0,1,1\}$ & \makecell{No} \\
\hline
\end{tabular}
\caption{Four-dimensional Frobenius Lie algebras}\label{4D}
\end{center}
\end{table}

\begin{remark}

Note that $\Phi'''(\varepsilon_1)=\Phi'''(\varepsilon_2)$ if and only if the quotient $\varepsilon_1/\varepsilon_2$ is the square of an element in $\mathbb{F},$ and all other pairs of Lie algebras in the table are non-isomorphic.
Only $\Phi^\prime$ and $\Phi^{\prime\prime}(\Delta)$ have 
deformations.  $\Phi'''(\eps)$ is organized as a parametrized family of Lie algebras, however, it cannot be represented as a formal deformation. 
For $\Phi'$, there are three deformations: $\Phi'_{1,t},\Phi'_{2,t},$ and $\Phi'_{3,t}$, with deformation parameter $t$. The latter two are jump deformations -- where the deformed algebras happens to be isomorphic to the initial algebra $\Phi'.$  The first deformation will become our prime example. For $\Phi''(\Delta),$ there are two deformations: $\Phi_{1,t}''(\Delta)$ and $\Phi_{2,t}''(0).$ The former exists for all values of $\Delta\in\mathbb{F},$ while the latter is a jump deformation and exists only when $\Delta=0.$ 

\end{remark}

%For $\Phi''(\Delta),$ there is one deformation corresponding to all $\Delta\in\mathbb{F}$ with deformation parameter $t$. There is another deformation of $\Phi''(\Delta)$ only in the particular case when $\Delta=0.$ This is a jump deformation with parameter $t=-\eps$ only when $\eps\in\mathbb{C}.$ When $\eps\not\in\mathbb{C},$ this is not a jump deformation, since $\Phi'''(\eps_1)$ is not necessarily isomorphic to $\Phi'''(\eps_2)$ for $\eps_i\not\in\mathbb{C}.$ $\Phi''_{2,t}(0)$ still deforms along $\Phi'''(\eps)$ in the sense that $t=-\eps.$ Thus, we may conclude (about the entire $\Phi''(\Delta)$ family) that there is a sort of ``bifurcation" of the deformations, i.e., $\Phi''(\Delta)$ deforms along itself along $\Delta'=\Delta+t$ in only one infinitesimal direction until $t=-\Delta,$ i.e., $\Delta'=0.$ At this time, $\Phi''(\Delta')$ gains cohomology, and thus another deformation, one which moves into the $\eps$ family. In the case of $\Phi'''(\eps),$ the dimension of the second cohomology group is zero, so there are no infinitesimal deformations. That is, $\Phi'''(\eps)$ is organized as a parametrized family of Lie algebras, however, it cannot be represented as a formal deformation.

\section{Deformations}\label{sec:deformation} 

\textit{Notation:} To ease notation, we let $\{e_1, e_2, e_3, e_4\}$ be a basis for the Lie algebra $\mathfrak{g}$ under consideration and use $\Gamma^i$ to represent an $i$-cocycle of $\mathfrak{g}$ with coefficients in the adjoint representation of $ \mathfrak{g}$.

\subsection{Cohomology}

In this section, we compute $H^2(\mathfrak{g},\mathfrak{g})$ and $H^3(\mathfrak{g},\mathfrak{g})$ for the Lie algebras $\mathfrak{g}$ listed in Table \ref{4D}.    We provide detailed calculations for $\Phi^\prime$ (see Theorems \ref{thm:phi1h2} and \ref{thm:zero}).  The calculations for the other algebras in Table 1 are similar.  

\begin{theorem}\label{thm:phi1h2}  
A basis for $H^2(\Phi^\prime,\Phi^\prime)$ is given by
$\{[\Gamma_{1,1}^2],[\Gamma_{1,2}^2],[\Gamma_{1,3}^2]\},$  where the $\Gamma$s are defined by 

 \begin{itemize}
        \item $\Gamma_{1,1}^2$ is defined by $\Gamma_{1,1}^2(e_2,e_4)=e_2,$ $\Gamma_{1,1}^2(e_3,e_4)=-e_3.$
        \item $\Gamma_{1,2}^2$ is defined by $\Gamma_{1,2}^2(e_3,e_4)=e_2.$
        \item $\Gamma_{1,3}^2$ is defined by $\Gamma_{1,3}^2(e_2,e_4)=e_3.$
    \end{itemize}
\end{theorem}

\noindent

\begin{proof}
The fact that $\Gamma^2$ is a cocycle gives the following conditions on its coefficients:\\

\vspace{0.1cm}

\begin{minipage}{0.3\linewidth}
\begin{itemize}
    \item $c_1^{1,2}=c_3^{2,3}-3c_4^{2,4}$
    \item $c_2^{1,2}=\frac{1}{2}c_4^{2,3}$
    \item $c_3^{1,2}=0$
    \item $c_4^{1,2}=0$
\end{itemize}
\end{minipage}
\begin{minipage}{0.3\linewidth}
\begin{itemize}
    \item $c_1^{1,3}=-c_2^{2,3}-3c_4^{3,4}$
    \item $c_2^{1,3}=0$
    \item $c_3^{1,3}=\frac{1}{2}c_4^{2,3}$
    \item $c_4^{1,3}=0$
\end{itemize}
\end{minipage}
\begin{minipage}{0.3\linewidth}
\begin{itemize}
    \item $c_1^{1,4}=c_2^{2,4}+c_3^{3,4}$
    \item $c_2^{1,4}=\frac{1}{2}c_2^{2,3}+\frac{1}{2}c_4^{3,4}$
    \item $c_3^{1,4}=\frac{1}{2}c_3^{2,3}-\frac{1}{2}c_4^{2,4}$
    \item $c_4^{1,4}=c_4^{2,3}.$
\end{itemize}
\end{minipage}

\vspace{0.3cm}

\noindent 
If $\Gamma^2(e_i,e_j)=\sum_{k=1}^4 c_k^{i,j}e_k$, $1<i<j<4$
is to be a coboundary, there must exist $F^1\in C^1(\mathfrak{g},\mathfrak{g})$ defined by $F^1(e_{\ell})=\sum_{k=1}^4 c_k^{\ell}e_k,$ $1<\ell<4$
such that the following conditions are satisfied:\\

\vspace{0.1cm}

\begin{minipage}{0.3\linewidth}
\begin{enumerate}
    \item $c_1^1=c_2^2+c_3^3+c_1^{2,3}$
    \item $c_3^4=-\frac{1}{2}c_1^2-c_1^{2,4}$
    \item $c_1^3=2c_2^4-2c_1^{3,4}$
    \item $c_2^1=\frac{1}{2}c_4^3+c_2^{2,3}$
\end{enumerate}
\end{minipage}
\begin{minipage}{0.3\linewidth}
\begin{enumerate}
    \item[5.] $c_3^1=-\frac{1}{2}c_4^2+c_3^{2,3}$
    \item[6.] $c_4^1=c_4^{2,3}$
    \item[7.] $c_4^2=2c_4^{2,4}$
    \item[8.] $c_4^3=2c_4^{3,4}$
\end{enumerate}
\end{minipage}
\begin{minipage}{0.3\linewidth}
\begin{enumerate}
    \item[9.] $c_4^4=-2c_2^{2,4}=-2c_3^{3,4}$
    \item[10.] $c_2^{3,4}=0$
    \item[11.] $c_3^{2,4}=0$.
\end{enumerate}
\end{minipage}

\vspace{0.3cm}

\noindent
Note that for a fixed $\Gamma^2$, we need only be careful of the choices for coefficients of $F^1$ as dictated by conditions 1-8.  However, conditions 9-11 put independent conditions on $\Gamma^2,$ rather than conditions on $F^1.$ This means that the $\Gamma^2$'s which do not satisfy $9-11$ have no such $F^1.$ This gives a upper bound of three on the dimension of $H^2(\Phi',\Phi').$
Choosing $\Gamma_{1,1}^2, \Gamma_{1,2}^2,$ and $\Gamma_{1,3}^2$ as above yields the theorem --  it is straightforward to verify that these are non-cohomologous cocycles. \end{proof}

%Now, to show that $\dim H^2(\Phi',\Phi')\leq 3,$ let $\Gamma^2_*\in C^2(\Phi',\Phi')$ be defined as $\Gamma^2_*(e_i,e_j)=\sum_{k=1}^4a_k^{i,j}e_k,$ $1<i<j<4$
%such that $a^{2,4}_2\neq a^{3,4}_3,$ $a^{3,4}_2\neq 0,$ and $a^{2,4}_3\neq 0.$ That is, we choose $\Gamma^2_*$ so that conditions 9, 10, and 11 are not satisfied.

%If we let $\alpha=\frac{1}{2}(a^{2,4}_2-a^{3,4}_3),$ $\beta=a^{3,4}_2,$ and $\gamma=a^{2,4}_3,$ then $\Gamma^2_*$ is cohomologous with $\alpha\Gamma^2_{1,1}+\beta\Gamma^2_{1,2}+\gamma\Gamma^2_{1,3}.$ Hence, $[\Gamma^2_{1,1}],[\Gamma^2_{1,2}],$ and $[\Gamma^2_{1,3}]$ forms a basis for $H^2(\Phi',\Phi').$

For $\Phi'$, the following theorem asserts that the the third cohomolgoy group is trivial, so the infinitesimals given in Theorem 1 are unobstructed.

\begin{theorem}\label{thm:zero}
$H^3(\Phi',\Phi')=0.$
\end{theorem}
\begin{proof}
Let $\Gamma^3(e_i,e_j,e_{\ell})=\sum_{k=1}^4c_k^{i,j,\ell}e_k,$ $1<i<j<\ell<4.$
By a straightforward computation, we see that, in order for $\Gamma^3$ to indeed be a cocycle, the following conditions on its coefficients must be satisfied:

\vspace{0.5cm}

\begin{minipage}{0.5\linewidth}
\begin{itemize}
    \item $c^{1,2,3}_1=-c^{1,2,4}_2-c^{1,3,4}_3+c^{2,3,4}_4$
    \item $3c^{1,2,3}_2=-c^{1,3,4}_4$
\end{itemize}
\end{minipage}
\begin{minipage}{0.5\linewidth}
\begin{itemize}
    \item $3c^{1,2,3}_3=c^{1,2,4}_4$
    \item $c^{1,2,3}_4=0.$
\end{itemize}
\end{minipage}

\vspace{0.5cm}

With these conditions established for $\Gamma^3\in Z^3(\Phi,\Phi),$ we will show that there exists some 2-cochain $\Gamma^2$ for which $\delta\Gamma^2=\Gamma^3.$
If $\Gamma^2(e_i,e_j)=\sum_{k=1}^4c_k^{i,j},$ $1<i<j<4,$ and

\vspace{0.5cm}

\begin{minipage}{0.3\linewidth}
\begin{itemize}
    \item $c^{1,2}_1=2(c^{1,4}_3-c^{2,4}_4-c^{1,2,4}_1)$
    \item $c^{1,2}_2=\frac{1}{2}c^{1,4}_4-c^{1,2,4}_2$
    \item $c^{1,2}_3=-c^{1,2,4}_3$
    \item $c^{1,2}_4=-\frac{2}{3}c^{1,2,4}_4$
\end{itemize}
\end{minipage}
\begin{minipage}{0.32\linewidth}
\begin{itemize}
    \item $c^{1,3}_1=-2(c^{1,4}_2+c^{3,4}_4+c^{1,3,4}_1)$
    \item $c^{1,3}_2=-c^{1,3,4}_2$
    \item $c^{1,3}_3=\frac{1}{2}c^{1,4}_4-c^{1,3,4}_3$
    \item $c^{1,3}_4=-\frac{2}{3}c^{1,3,4}_4$
\end{itemize}
\end{minipage}
\begin{minipage}{0.3\linewidth}
\begin{itemize}
    \item $c^{1,4}_1=c^{2,4}_2+c^{3,4}_3+c^{2,3,4}_1$
    \item $c^{2,3}_2=2c^{1,4}_2-c^{3,4}_4-2c^{2,3,4}_2$
    \item $c^{2,3}_3=2c^{1,4}_3+c^{2,4}_4-2c^{2,3,4}_3$
    \item $c^{2,3}_4=c^{1,4}_4-c^{2,3,4}_4.$
\end{itemize}
\end{minipage}

\vspace{0.5cm}

then $\delta \Gamma^2=\Gamma^3,$ and the result follows.
\end{proof}

\subsection{Deformation and Spectrum}\label{defspec}

The main results of this section are displayed in the table below.

\begin{table}[H]
\begin{center}
\renewcommand\arraystretch{3}
\setlength\doublerulesep{0pt}
\begin{tabular}{|p{1.7cm}|||p{2cm}|p{4.2cm}|p{5cm}|}
\hline
\centering $\mathfrak{g}_t$ & \centering Infinitesimals (See Theorem~\ref{thm:phi1h2}) & \centering Commutator Relations & \makecell{Spectrum} \\
 \hline\hline\hline
 \centering $\Phi'_{1,t}$ & \centering$\Gamma_{1,1}^2$ & \makecell[l]{$[e_1,e_4]_t=[e_2,e_3]_t=-e_1,$ \\ $[e_2,e_4]_t=(t-\frac{1}{2})e_2,$ \\ $[e_3,e_4]_t=-(t+\frac{1}{2})e_3$} & \makecell{$\left\{0,\frac{1}{2}-t,\frac{1}{2}+t,1\right\}$}\\
\hline
 \centering $\Phi'_{2,t}$ & \centering$\Gamma_{1,2}^2$ & \makecell[l]{$[e_1,e_4]_t=[e_2,e_3]_t=-e_1,$ \\ $[e_2,e_4]_t=-\frac{1}{2}e_2,$ \\  $[e_3,e_4]_t=-\frac{1}{2}e_3+te_2$} & \makecell{$\left\{0,\frac{1}{2},\frac{1}{2},1\right\}$}\\
 \hline
  \centering $\Phi'_{3,t}$ & \centering$\Gamma_{1,3}^2$ & \makecell[l]{$[e_1,e_4]_t=[e_2,e_3]_t=-e_1,$ \\ $[e_2,e_4]_t=-\frac{1}{2}e_2+te_3,$ \\  $[e_3,e_4]_t=-\frac{1}{2}e_3$} & \makecell{$\left\{0,\frac{1}{2},\frac{1}{2},1\right\}$}\\
  \hline\hline\hline
 \centering \makecell{$\Phi''_{1,t}(\Delta)$} & \centering$\Gamma_{2,1}^2$ &  \makecell[l]{$[e_1,e_4]_t=[e_2,e_3]_t=-e_1,$ \\ $[e_2,e_4]_t=-e_3,$ \\ $[e_3,e_4]_t=(t+\Delta)e_2-e_3$} &  \makecell{$\left\{0,1,\frac{1\pm\sqrt{1-4(\Delta+t)}}{2}\right\}$} \\
\hline
 \centering $\Phi''_{2,t}(0)$ & \centering$\Gamma_{2,2}^2$ &  \makecell[l]{$[e_1,e_2]_t=te_3,$ \\ $[e_1,e_4]_t=[e_2,e_3]_t=-e_1,$ \\ $[e_2,e_4]_t=[e_3,e_4]_t=-e_3,$} &  \makecell{$\{0,0,1,1\}$} \\
\hline
\end{tabular}
\caption{Deformations of four- dimensional Frobenius Lie algebras}
\label{table:def}
\end{center}
\end{table}

\begin{remark}
For given $\Delta\in\mathbb{F},$ $\Phi'_{1,t}\cong \Phi''(\Delta)$ when $t=\frac{\sqrt{1-4\Delta}}{2}.$ This follows from comparing the spectra of the respective Lie algebras and the fact that the category of Frobenius Lie algebras is stable under deformation.
\end{remark}

\begin{remark}
Fixing a value of $\Delta\in\mathbb{F}$ and replacing the basis $\{e_1,e_2,e_3,e_4\}$ by $\{e_1'=2te_1,e_2'=-(t+\frac{1}{2})e_2+e_3,e_3'=(t-\frac{1}{2})e_2+e_3,e_4\},$ we see that $\Phi''(\Delta)\cong \Phi'_{1,t}$ for $t=\frac{\sqrt{1-4\Delta}}{2}\neq 0.$ Replacing the basis $\{e_1,e_2,e_3,e_4\}$ by $\{e_1,e_2,e'_3=e_3+2te_2,e_4\}$ we see that $\Phi'\cong\Phi'_{2,t}$. Replacing the basis $\{e_1,e_2,e_3,e_4\}$ by $\{e_1,e'_2=e_2+2te_3,e_3,e_4\}$ we see that $\Phi'\cong\Phi'_{3,t}$.  Since $F_1$ and $F_2$ are non-cohomologous, they define inequivalent jump deformations. Replacing the basis $\{e_1,e_2,e_3,e_4\}$ by $\{e_1,e_2'=e_2-e_4,e_3'=e_2,e_4'=e_3\},$ we see that $\Phi'''(\eps)\cong \Phi''_{2,t}(0)$ for $t=-\eps.$
\end{remark}

\begin{remark}
We will only prove the results from the table corresponding to the Lie algebra $\Phi'_{1,t}.$ The proofs of the other cases are similar.
\end{remark}

\begin{theorem}
$\Gamma^2_{1,1}$ is the infinitesimal of a deformation of $\Phi'$, giving rise to the deformed algebra $\Phi'_{1,t}$ defined by the relations

\vspace{0.5cm}

\begin{minipage}{0.3\linewidth}
\begin{itemize}
    \item $[e_1,e_4]_t=[e_2,e_3]_t=-e_1$
\end{itemize}
\end{minipage}
\begin{minipage}{0.3\linewidth}
\begin{itemize}
    \item $[e_2,e_4]_t=(t-\frac{1}{2})e_2$ 
\end{itemize}
\end{minipage}
\begin{minipage}{0.3\linewidth}
\begin{itemize}
    \item $[e_3,e_4]_t=-(t+\frac{1}{2})e_3$.
\end{itemize}
\end{minipage}
\end{theorem}

\begin{proof}
Set $\alpha_1=\Gamma^2_{1,1}\in H^2(\Phi',\Phi')$; that is, 

\vspace{0.5cm}

\begin{center}
\begin{minipage}{0.3\linewidth}
\begin{itemize}
\centering
    \item $\alpha_1(e_2,e_4)=e_2$ ~~~~and
\end{itemize}
\end{minipage}
\begin{minipage}{0.3\linewidth}
\begin{itemize}
    \item $\alpha_1(e_3,e_4)=-e_3.$
\end{itemize}
\end{minipage}
\end{center}

\vspace{0.5cm}
\noindent
Recall, that if $\alpha_1$ is the infinitesimal of a deformation $\sum_{i\ge 0}\alpha_it^i$, then $\delta\alpha_2=-\frac{1}{2}[\alpha_1,\alpha_1]$. Calculating,

\vspace{0.2cm}

\begin{itemize}
\item $[\alpha_1,\alpha_1](e_1,e_2,e_3)=2\alpha_1(\alpha_1(e_1,e_2),e_3)-2\alpha_1(\alpha_1(e_1,e_3),e_2)+2\alpha_1(\alpha_1(e_2,e_3),e_1)=0$
\item $[\alpha_1,\alpha_1](e_1,e_2,e_4)=2\alpha_1(\alpha_1(e_1,e_2),e_4)-2\alpha_1(\alpha_1(e_1,e_4),e_2)+2\alpha_1(\alpha_1(e_2,e_4),e_1)=2\alpha_1(e_2,e_1)=0$
\item $[\alpha_1,\alpha_1](e_1,e_3,e_4)=2\alpha_1(\alpha_1(e_1,e_3),e_4)-2\alpha_1(\alpha_1(e_1,e_4),e_3)+2\alpha_1(\alpha_1(e_3,e_4),e_1)=-2\alpha_1(e_3,e_1)=0$
\item $[\alpha_1,\alpha_1](e_2,e_3,e_4)=2\alpha_1(\alpha_1(e_2,e_3),e_4)-2\alpha_1(\alpha_1(e_2,e_4),e_3)+2\alpha_1(\alpha_1(e_3,e_4),e_2)$$ $$=-2\alpha_1(e_2,e_3)-2\alpha_1(e_3,e_2)\\ ~~~~~~~~~~~~~~~~~~~~~~~~=0.$
\end{itemize}

\vspace{0.3cm}

\noindent
Thus, $\delta\alpha_2=0$ which implies that we get a deformation of  $\Phi'_{1,t}$ defined as follows:
\\*

\noindent

\begin{minipage}{0.3\linewidth}
\begin{itemize}
\item $[e_1,e_4]_t=[e_2,e_3]_t=-e_1$
\end{itemize}
\end{minipage}
\begin{minipage}{0.3\linewidth}
\begin{itemize}
\item $[e_2,e_4]_t=(t-\frac{1}{2}e_2)$
\end{itemize}
\end{minipage}
\begin{minipage}{0.3\linewidth}
\begin{itemize} 
\item $[e_3,e_4]_t=-(t+\frac{1}{2})e_3.$
\end{itemize}
\end{minipage}
\end{proof}

\begin{theorem}
The spectrum of $\Phi'_{1,t}$ is given by $\{0,\frac{1}{2}-t, \frac{1}{2}+t,1\}$.
\end{theorem}
\begin{proof}
To determine the spectrum of $\Phi'_{1,t}$, we must first determine a choice of Frobenius functional. Let $F=f_1e^*_1+f_2e^*_2+f_3e_3^*+f_4e_4^*$ and $B=b_1e_1+b_2e_2+b_3e_3+b_4e_4\in \Phi'_{1,t}\cap\ker(F)$. For $F$ to be Frobenius, it must satisfy the system following system of equations:

\vspace{0.5cm}

\begin{minipage}{0.5\linewidth}
\begin{itemize}
    \item $F([e_1,B])=-b_4f_1=0$
    \item $F([e_2,B])=-b_3f_1-\frac{1}{2}b_4f_2=0$
\end{itemize}
\end{minipage}
\begin{minipage}{0.5\linewidth}
\begin{itemize}
    \item $F([e_3,B])=b_2f_1-\frac{1}{2}b_4f_3=0$ 
    \item $F([e_4,B])=b_1f_1+\frac{1}{2}b_2f_2+\frac{1}{2}b_3f_3=0$
\end{itemize}
\end{minipage}

\vspace{0.5cm}

\noindent
must imply that $B=0$; this is accomplished by taking $f_1=1$ and $f_2=f_3=f_4=0$. Thus, $F=e_1^*$ is a Frobenius functional on $\Phi'_{1,t}$. 

Next, we need to determine the principal element $\widehat{F}\in\Phi'_{1,t}$ corresponding to $F$. If $\widehat{F}=p_1e_1+p_2e_2+p_3e_3+p_4e_4$, then it must be the case that

\vspace{0.5cm}

\begin{minipage}{0.5\linewidth}
\begin{itemize}
    \item $p_4=F([\widehat{F},e_1])=F(e_1)=1$
    \item $p_3=F([\widehat{F},e_2])=F(e_2)=0$
\end{itemize}
\end{minipage}
\begin{minipage}{0.5\linewidth}
\begin{itemize}
    \item $-p_2=F([\widehat{F},e_3])=F(e_3)=0$
    \item $-p_1=F([\widehat{F},e_4])=F(e_4)=0$.
\end{itemize}
\end{minipage}

\vspace{0.5cm}

\noindent
We conclude that $\widehat{F}=e_4$. 

Finally, to determine the spectrum of $\Phi'_{1,t}$, we calculate the spectrum of $[\widehat{F},-]:\Phi'_{1,t}\to \Phi'_{1,t}$. It is straightforward to show that

$$\textup{ad}\widehat{F}=\begin{bmatrix}
   1 & 0 & 0 & 0  \\
   0 & \frac{1}{2}-t & 0 & 0 \\
   0 & 0 & \frac{1}{2}+t & 0 \\
    0 & 0 & 0 & 0
\end{bmatrix}$$ so that the spectrum of $\Phi'_{1,t}$ is $$\left\{0,\frac{1}{2}-t, \frac{1}{2}+t,1\right\}.$$
\end{proof}

\section{Epilogue}

Topical investigations regarding the spectrum of a Frobenius Lie algebra have concentrated on  seaweed Lie algebras (see \textbf{\cite{DK}}), or simply ``seaweeds''  (elsewhere called \textit{biparabolic} \textbf{\cite{Joseph}}), and the recently introduced Lie poset algebras (see \textbf{\cite{CG}}).  In a series of papers 
by Coll et al (see \textbf{\cite{unbroken},\cite{specAB},} and \textbf{\cite{specD}}), it has been established that the unbroken spectrum property holds for all the classical and exceptional Frobenius seaweeds.

Indeed, the interesting spectral properties of seaweeds was the impetus for the motivating question of this article.  However, seaweeds appear to be 
 cohomologically inert so cannot be deformed.\footnote{In 2014 Gerstenhaber conjectured that seaweeds were cohomologically trivial.  This was verified for type-A seaweeds by Elashvili and Rakviashvili (see \textbf{\cite{elashvili}, 2016)}. Subsequent work (yet to be published) by Coll et al (see \textbf{\cite{CHM}})
 establishes that for types B, C, and D seaweeds are cohomologically trivial.  This latter work is the central result of Alan Hylton's Ph.D. dissertation at Lehigh University (in progress).}

In contrast, Lie poset algebras, which are necessarily solvable, have a rich deformation theory.  However, we have no examples of deformable Frobenius Lie poset algebras.  
It is also worth noting that the unbroken spectrum property seems to be a property of Frobenius Lie poset algebras, although the spectrum is ``binary'', consisting of only 0's and 1's (see \textbf{\cite{CM},\cite{CM2},\cite{CM3},} and \textbf{\cite{CMR}}).

It is interesting to note that the spectrum of $\Phi''(\Delta)$ is an unbroken sequence of integers if and only if $\Delta=0$ or -2. The proof is straightforward as follows.  Recall that the spectrum of $\Phi''(\Delta)$ is $\{0,1,\frac{1\pm\sqrt{1-4\Delta}}{2}\}$. Thus, the spectrum of $\Phi''(\Delta)$ consists of integers if and only if $1-4\Delta=a^2$, where $a$ is an odd integer. If $1-4\Delta=a^2$, then the spectrum of $\Phi''(\Delta)$ is given by $\{0,1,\frac{1\pm a}{2}\}$. When $a=\pm 1$, i.e., $\Delta=0$, the spectrum of $\Phi''(\Delta)$ is $\{0,0,1,1\}$; and when $a=\pm 3$, i.e., $\Delta=-2$, the spectrum of $\Phi''(\Delta)$ is $\{-1,0,1,2\}$. Since $\frac{1-a}{2}$ (resp. $\frac{1+a}{2}$) strictly decreases (resp. increases) as $a$ increases, the spectrum can only be unbroken for $a=\pm 1$ and $\pm 3$, i.e., $\Delta=0$ and -2.

Note that $\Phi''(\Delta)$, where 
$\Delta = 0, -2$ is not a seaweed, since it deforms. 
And while $\Phi''(0)$ and $\Phi''(-2)$  are both solvable, neither is a Lie poset algebra of classical type since there is exactly one such four-dimensional Frobenius algebra. When the ground field is the complex numbers this algebra is isomorphic to $\Phi^{'''}(\epsilon)$, for all $\epsilon \ne 0$.

%Note also
%that $\Phi^{''}(\Delta)$, for $\Delta = 0$ has a binary spectrum which is suggestive
%though not definitive that it is a Frobenius Lie poset algebra of exceptional type.

%\begin{remark}

%The unbroken spectrum at $\Delta=0$
% and  $\Delta=-2$ begs the question of whether or not these are seaweeds or Lie poset algebras.
 % \begin{enumerate}
    % \item $\Delta=0$
   %   Not a seaweed since it deforms. Not a type-A Lie poset algebra by a low dimension argument (there is only one in type A, the epsilon just the carrot symbol).
     %     \item $\Delta=-2$
  %   Not a seaweed since it deforms. Not a poset since it does not have binary spectrum.
     
 %\end{enumerate}

%\end{remark}

\section{Appendix A -- dimension six}

The following table contains information about the isomorphism classes of non-decomposable 6-dimensional Frobenius Lie algebras over an algebraically closed field of characteristic 0.
\begin{center}
\renewcommand\arraystretch{2}
\setlength\doublerulesep{0pt}
\begin{tabular}{|p{1.8cm}|||p{4.4cm}|p{1cm}|p{1cm}|p{2.1cm}|p{2.2cm}|}
\hline
$\dim=6$ & Commutator Relations & dim $H^2$ & dim $H^3$ & Spectrum & Deformation(s)\\
 \hline\hline\hline
$\Phi_{6,1}$ & $[X_1,Y_1]=Y_1,$ $[X_1,Y_3]=Y_3,$  $[X_1,Y_4]=2Y_4,$ $[X_2,Y_2]=Y_2,$ $[X_2,Y_3]=Y_3,$ $[X_2,Y_4]=Y_4,$ $[Y_1,Y_2]=Y_3,$ $[Y_1,Y_3]=Y_4$ & 0 & 0 & $0,0,0,$ $1,1,1$ & No \\
\hline
$\Phi_{6,2}\{\xi,\eta\},$ $\xi\neq \eta$ & $[X_1,Y_1]=Y_1,$ $[X_1,Y_3]=Y_3,$ $[X_1,Y_4]=\xi Y_4,$ $[X_2,Y_2]=Y_2,$ $[X_2,Y_3]=Y_3,$ $[X_2,Y_4]=\eta Y_4,$ $[Y_1,Y_2]=Y_3$ & & & $0,0,\frac{\eta-1}{\eta-\xi},$ $\frac{1-\xi}{\eta-\xi}, 1,1$ & \\
\hline
$\Phi_{6,3}\{\xi:\eta\},$ $(\xi:\eta)\neq (1:1)$ & $[X_1,Y_1]=Y_1,$ $[x_1,Y_3]=Y_3,$ $[X_1,Y_4]=Y_4+\xi Y_3,$ $[X_2,Y_2]=Y_2,$ $[X_2,Y_3]=Y_3,$ $[X_2,Y_4]=Y_4+\eta Y_3,$ $[Y_1,Y_2]=Y_3$ & & & $0,0,\frac{\eta}{\eta-\xi},$ $\frac{\eta}{\xi-\eta},0,0$ & \\
\hline
$\Phi_{6,4}(\xi:\eta),$ $(\xi:\eta)\neq (0:0)$ & $[X_1,Y_1]=Y_1+\xi Y_4,$ $[X_1,Y_3]=Y_3,$ $[X_1,Y_4]=Y_4,$ $[X_2,Y_1]=\eta Y_4,$ $[X_2,Y_2]=Y_2,$ $[X_2,Y_3]=Y_3,$ $[Y_1,Y_2]=Y_3$ & & & $0,0,0,$ $1,1,1$ & \\
\hline
$\Phi_{6,5}(\xi:\eta),$ $\eta\neq 0$ & $[X_1,Y_1]=\frac{1}{2}Y_1+\xi Y_2,$ $[X_1,Y_2]=\frac{1}{2}Y_2,$ $[X_1,Y_3]=Y_3,$ $[X_2,Y_1]=\eta Y_2,$ $[X_2,Y_4]=Y_4,$ $[Y_1,Y_2]=Y_3$ & & & $0,0,\frac{1}{2},$ $\frac{1}{2},1,1$ & \\
\hline
$\Phi_{6,6}$ & $[X,Y_1]=Y_1,$ $[X,Y_2]=2Y_2,$ $[X,Y_3]=3Y_3,$ $[X,Y_4]=4Y_4,$ $[X,Y_5]=5Y_5,$ $[Y_1,Y_2]=Y_3,$ $[Y_1,Y_3]=Y_4,$ $[Y_1,Y_4]=Y_5,$ $[Y_2,Y_3]=Y_5$ & 0 & 0 & $0,\frac{1}{5},\frac{2}{5},$ $\frac{3}{5},\frac{4}{5},1$ & No \\
\hline
$\Phi_{6,7}(\xi)$ & $[X,Y_1]=\xi Y_1,$ $[X,Y_2]=2\xi Y_2,$ $[X,Y_3]=(1-2\xi)Y_3,$ $[X,Y_4]=(1-\xi)Y_4,$ $[X,Y_5]=Y_5,$ $[Y_1,Y_3]=Y_4,$ $[Y_1,Y_4]=[Y_2,Y_3]=Y_5$ & & & $0,\xi,2\xi,$ $1-~2\xi,1~-~\xi,~1$ & \\
\hline
$\Phi_{6,8}$ & $[X,Y_1]=Y_2,$ $[X,Y_3]=Y_3-Y_4,$ $[X,Y_4]=Y_4,$ $[X,Y_5]=Y_5,$ $[Y_1,Y_3]=Y_4,$ $[Y_1,Y_4]=[Y_2,Y_3]=Y_5$ & 3 & 2 & $0,0,0$ $1,1,1$ & Yes (3) \\
\hline
\end{tabular}
\end{center}
\begin{center}
\begin{table}[H]
\renewcommand\arraystretch{2}
\setlength\doublerulesep{0pt}
\begin{tabular}{|p{1.8cm}|||p{4.4cm}|p{1cm}|p{1cm}|p{2.1cm}|p{2.2cm}|}
\hline
$\Phi_{6,9}$ & $[X,Y_1]=\frac{1}{3}Y_1+Y_3,$ $[X,Y_2]=\frac{2}{3}Y_2+Y_4,$ $[X,Y_3]=\frac{1}{3}Y_3,$ $[X,Y_4]=\frac{2}{3}Y_4,$ $[X,Y_5]=Y_5,$ $[Y_1,Y_3]=Y_4,$ $[Y_1,Y_4]=[Y_2,Y_3]=Y_5$ & 1 & 0 & $0,\frac{1}{3},\frac{1}{3},$ $\frac{2}{3},\frac{2}{3},1$ & Yes (1)\\
\hline
$\Phi_{6,10}$ & $[X,Y_1]=\frac{1}{4}Y_1,$ $[X,Y_2]=\frac{1}{2}Y_2,$ $[X,Y_3]=\frac{1}{2}Y_3+Y_2,$ $[X,Y_4]=\frac{3}{4}Y_4,$ $[X,Y_5]=Y_5,$ $[Y_1,Y_3]=Y_4,$ $[Y_1,Y_4]=[Y_2,Y_3]=Y_5$ & 1 & 0 & $0,\frac{1}{4},\frac{1}{2},$ $\frac{1}{2},\frac{3}{4},1$ & Yes (1)\\
\hline
$\Phi_{6,11}\{\xi,\eta\}$ & $[X,Y_1]=\xi Y_1$ $[X,Y_2]=\eta Y_2,$ $[X,Y_3]=(1-\eta)Y_3,$ $[X,Y_4]=(1-\xi)Y_4,$ $[X,Y_5]=[Y_1,Y_4]=[Y_2,Y_3]=Y_5$ & & & $0,\xi,\eta,$ $1-\xi,$ $1-\eta,1$ & \\
\hline
$\Phi_{6,12}(\xi)\cong\Phi_{6,12}(1-\xi)$ & $[X,Y_1]=\xi Y_1,$ $[X,Y_2]=\frac{1}{2}Y_2+Y_3,$ $[X,Y_3]=\frac{1}{2}Y_3,$ $[X,Y_4]=(1-\xi)Y_4,$ $[X,Y_5]=[Y_1,Y_4]=[Y_2,Y_3]=Y_5$ & & & $0,\frac{1}{2},\frac{1}{2},\xi,\text{\ \ }$ $1-\xi,1$ & \\
\hline
$\Phi_{6,13}(\xi)\cong \Phi_{6,13}(1-\xi)$ & $[X,Y_1]=\xi Y_1+Y_2,$ $[X,Y_2]=\xi Y_2,$ $[X,Y_3]=(1-\xi)Y_3-Y_4,$ $[X,Y_4]=(1-\xi)Y_4,$ $[X,Y_5]=[Y_1,Y_4]=[Y_2,Y_3]=Y_5$ & & & $0,\xi,\xi,$ $1-\xi,1-~\xi,~1$ & \\
\hline
$\Phi_{6,14}$ & $[X,Y_1]=\frac{1}{2}Y_1+Y_2,$ $[X,Y_2]=\frac{1}{2}Y_2+Y_3,$ $[X,Y_3]=\frac{1}{2}Y_3-Y_4,$ $[X,Y_4]=\frac{1}{2}Y_4,$ $[X,Y_5]=[Y_1,Y_4]=[Y_2,Y_3]=Y_5$ & 2 & 0 & $0,\frac{1}{2},\frac{1}{2},$ $\frac{1}{2},\frac{1}{2},1$ & Yes (2)\\
\hline
$\Phi_{6,15}$ & $[X,Y_1]=\frac{1}{2}Y_1+Y_4,$ $[X,Y_2]=\frac{1}{2}Y_2+Y_3,$ $[X,Y_3]=\frac{1}{2}Y_3,$ $[X,Y_4]=\frac{1}{2}Y_4,$ $[X,Y_5]=[Y_1,Y_4]=[Y_2,Y_3]=Y_5$ & 4 & 0 & $0,\frac{1}{2},\frac{1}{2},$ $\frac{1}{2},\frac{1}{2},1$ & Yes (4) \\
\hline
\end{tabular}
\caption{Six-dimensional Frobenius Lie algebras}
\label{table:six}
\end{table}
\end{center}
If the parameters $\xi,\eta\in\mathbb{F}$ are separated by a colon, the isomorphism class does not change if both parameters are multiplied by a non-zero number. There are curly brackets around the parameters when the isomorphism class does not depend on the order of the parameters. The algebra $\Phi_{6,11}\{\xi,\eta\}$ depends only on the set $\{\xi,1-\xi,1-\eta\}.$ Except for these isomorphisms, the isomorphism classes of listed in the table are pairwise distinct.

\begin{remark}
The reader will have no difficulty establishing that all deformations associated with Table \ref{table:six} are linear.
\end{remark}

\begin{Ex}\label{ex:6dim}
Consider the family given by $\Phi_{6,12}(\xi).$ 
Routine calculation yield:
$$\dim H^2(\Phi_{6,12}(\xi),\Phi_{6,12}(\xi))=\begin{cases}
4, & \xi=0,1\\
3, & \xi=\frac{1}{4},\frac{3}{4}\\
6, & \xi=\frac{1}{2}\\
2, & \text{otherwise},
\end{cases}$$
and
$$\dim H^3(\Phi_{6,12}(\xi),\Phi_{6,12}(\xi))=\begin{cases}
2, & \xi=-1,0,1,2\\
1, & \xi=-\frac{1}{4},\frac{1}{4},\frac{3}{4},\frac{5}{4}\\
0, & \text{otherwise}.
\end{cases}$$
Note that the number of deformations of $\Phi_{6,12}(\xi)$ corresponds exactly to $\dim H^2$ for each value of $\xi,$ and each deformation is linear.
We illustrate one such deformation as follows.

Let $\Phi_{6,12,t}(2)$ be one of the two deformations given by $\Phi_{6,12}(\xi)$ when $\xi=2.$ This deformation family is defined by the following collection of commutator relations:

\vspace{0.5cm}

\begin{minipage}{0.5\linewidth}
\begin{itemize}
    \item $[e_1,e_2]_t=(2+\frac{t}{7})e_2$
    \item $[e_1,e_3]_t=(\frac{1}{2}-\frac{3t}{7})e_3+e_4$
    \item $[e_1,e_4]_t=(\frac{1}{2}-\frac{3t}{7})e_4$
\end{itemize}
\end{minipage}
\begin{minipage}{0.5\linewidth}
\begin{itemize}
    \item $[e_1,e_5]_t=(-1-t)e_5$
    \item $[e_1,e_6]_t=(1-\frac{6t}{7})e_6$
    \item $[e_2,e_5]_t=e_6$
    \item $[e_3,e_4]_t=e_6$
\end{itemize}
\end{minipage}

\end{Ex}

\end{document}